\documentclass[12pt, a4paper]{article}

\usepackage[latin1]{inputenc}
\usepackage[english]{babel}
\usepackage{mathtext,amssymb,amsmath}
\usepackage{amsthm,amsfonts,mathrsfs}
\usepackage{color,graphicx}
\usepackage{psfrag}
\usepackage{float}

\author{F\'abio S. Souza and Paul A. Schweitzer, S.J.}
\title{Manifolds that are not leaves of codimension one foliations}

\newtheorem{theo}{Theorem}[section]
\newtheorem{exe}[theo]{Example}

\newtheorem{defi}[theo]{Definition}

\newtheorem{prop}[theo]{Proposition}

\newtheorem{lemma}[theo]{Lemma}

\newtheorem*{theoA}{Theorem A}
\newtheorem*{theoB}{Theorem B}
\newtheorem*{theoC}{Theorem C}

\newcommand{\Nset}{\mathbb{N}}
\newcommand{\Rset}{\mathbb{R}}
\newcommand{\Sset}{\mathbb{S}}
\newcommand{\Zset}{\mathbb{Z}}
\newcommand{\Fol}{\mathscr{F}}
\newcommand{\Nfol}{\mathscr{N}}

\begin{document}

\maketitle

\begin{abstract}
We present new open manifolds that are not homeomorphic to leaves of any $C^0$ codimension one foliation of a compact manifold. Among them are simply connected manifolds of dimension $d\geq 5$ that are non-periodic in homotopy, namely in their 2-dimensional homotopy groups.
\medskip

\noindent{\bf keywords:} foliations,
non-leaves, non-periodic manifolds, codimension one foliations
\medskip

\end{abstract}

\tableofcontents


\section{Introduction}

In this paper we construct a large class of open (non-compact) manifolds
that cannot be leaves of codimension one foliations on compact
manifolds (see Theorems A, B, and C).
They are constructed as ``sum-manifolds" (see Definition \ref{summanifold}) by forming connected
sums of closed connected manifolds (called ``blocks"), replacing each vertex
of an infinite tree by a block and performing a connected
sum for each edge.

These non-leaves include simply connected open
manifolds that we call {\bf non-periodic in homotopy} (or in homology,
see Definition \ref{nonperiodic} for the precise definitions), since the second (or higher) homotopy group (or homology group) is not periodic at infinity.
These non-periodic sum-manifolds give the first examples of simply-connected $5$-manifolds
that cannot be leaves in codimension one. There are uncountably
many homotopy types of such non-leaves.

\begin{prop} For every $d\geq 5$ there exist simply connected $d$-dimensional sum-manifolds that are non-periodic in homotopy
and others that are non-periodic in homology. \label{existnonper}
\end{prop}

\begin{theoB}\label{theoremB}
Sum-manifolds that are non-periodic in homotopy (or in homology) are not homeomorphic to any leaf of a foliation of codimension one on a compact manifold.
\end{theoB}

This Theorem follows from Theorem A, which
includes and generalizes the results of Ghys \cite{Gh} and
Inaba {\em et al.} \cite{INTT}, who used the fundamental group
of certain open manifolds of dimension $d\geq 3$ to show that they
cannot be homeomorphic to leaves of any codimension one foliation
of a compact manifold.

\bigskip
\noindent{\bf History.} The problem of deciding when a manifold is homeomorphic or diffeomorphic to a leaf of a foliation (the ``realizability problem") is well-known. It was initially proposed by J. Sondow in 1975 \cite{So}. Since every manifold is a leaf in a product foliation, one considers non-compact manifolds and asks whether they can be leaves in foliations of compact manifolds. Cantwell and Conlon \cite{CC} showed that every open surface is diffeomorphic to a leaf of a foliation on a compact $3$-manifold, and as already mentioned,
Ghys \cite{Gh} and Inaba, Nishimori, Takamura, and Tsuchiya
 \cite{INTT}
showed that this does not hold in higher dimensions. Attie and Hurder \cite{AH}, using the first Pontrjagin class, constructed simply connected examples of dimension $6$ or greater that are not leaves in codimension one. There are also various constructions of
open Riemannian manifolds that are not quasi-isometric to
leaves of foliations on a compact manifold \cite{AH, Z, Sc1, Sc2}.

In Section \ref{secblocks} we give several definitions and outline the construction of sum-manifolds that are non-periodic in homotopy (and in homology) using Proposition \ref{existblocks}, which is proven by
constructing certain blocks in Section \ref{pconstruct}. The proof of our main result, Theorem A,
is given in Section \ref{nonleafproof} using two lemmas proven in Section \ref{secblocks}. The proof follows the proof of Ghys \cite{Gh}
in the initial steps, but the final part of the proof is
considerably simpler. The precise definition of manifolds
that are non-periodic in homotopy or homology is given in
Section \ref{nonpersec}. In this Section and the following one
it is shown that Theorem A implies Theorem B. Theorem C (in
Section \ref{secblocks}) is a version of Theorem A that
applies directly to the construction of sum-manifolds.

Some interesting open questions remain. Is it possible for a leaf in a codimension one foliation of a compact manifold to have an isolated non-periodic end? Are there any simply connected manifolds of dimension $3$ or $4$ that are not homeomorphic to leaves in codimension one foliations of a compact manifold?

This work is based on part of the doctoral thesis \cite{Sou} of the first author at PUC-Rio de Janeiro under the supervision of the second author.


\section{Blocks, sum-manifolds, and results}\label{secblocks}

\begin{defi} A {\bf block} (of dimension $d$) is a compact connected $d$-dimen\-sional manifold without boundary. If $B$ is a block, then a {\bf deleted} $B$-{\bf block} is a manifold homeomorphic to $B$ less the interiors of finitely many disjoint collared closed $d$-balls. A block is {\bf trivial} if it is homeomorphic to a sphere, and it is {\bf prime} if it cannot be decomposed as a connected sum of two non-trivial blocks, or, equivalently, if
every collared embedded $(d-1)$-sphere bounds a $d$-ball.
\label{blocks}
\end{defi}

\begin{prop} Every closed connected manifold is homeomorphic to a connected sum of a finite number of prime blocks. \label{indecsum}
\end{prop}

The proof of this Proposition is given in Section
\ref{primeblocks}. Note that the decomposition is not unique for any
dimension $d\geq 2$, since the connected sum of a projective plane and a torus is homeomorphic to the connected sum of a projective
plane with a Klein bottle and this example generalizes to higher
dimensions. On the other hand, when we consider oriented
manifolds (so that the connected sum operation is well defined)
Milnor \cite{M} has shown that the decomposition is unique for oriented $3$-manifolds, but not for oriented $4$-manifolds.

\begin{exe}{\rm
(\cite{M}, p. 6) If $P$ is the complex projective plane with the
usual orientation and $P'$ is the same manifold with the opposite
orientation,
then $P\# P'\# P'$ is diffeomorphic to $(S^2\times S^2)\#
P'$, while $S^2\times S^2$ is not homeomorphic to $P$ or $P'$, and
they are prime $4$-dimensional blocks.}
\label{milnorex}
\end{exe}

Now let $B_n$ be disjoint $d$-dimensional blocks, $n=1,2,\dots$, and let $G$ be an infinite connected graph (always supposed to be locally finite) whose set of vertices is $\{v_n\}_{n\in\Nset}$ and whose set of edges is $\{e_m\}_{m\in \Nset}$. For each edge $e_m$ of $G$ with endpoints $v_{n_1}$ and $v_{n_2}$, perform a connected sum operation between $B_{n_1}$ and $B_{n_2}$ to obtain a non-compact connected manifold $W$ as in \cite{INTT}. An analogous construction can be done using a finite graph.

\begin{defi} The manifold $W$ constructed in this way is a
{\bf sum-manifold patterned on the graph} $G$ (See Figure \ref{fig:sum}).
\label{summanifold}\end{defi}

We recall that a tree is a connected and simply connected locally finite graph and
observe that the examples of codimension one non-leaves given by
Ghys \cite{Gh} and
Inaba {\it et al.} \cite{INTT} are sum-manifolds patterned on infinite trees.

\begin{figure}[H]
\centering
 \psfrag{V1}{$B_3$}
 \psfrag{V2}{$B_1$}
 \psfrag{V3}{$B_2$}
 \psfrag{a}{$T$}
 \psfrag{W}{$W$}
\includegraphics*[width=.9\linewidth]{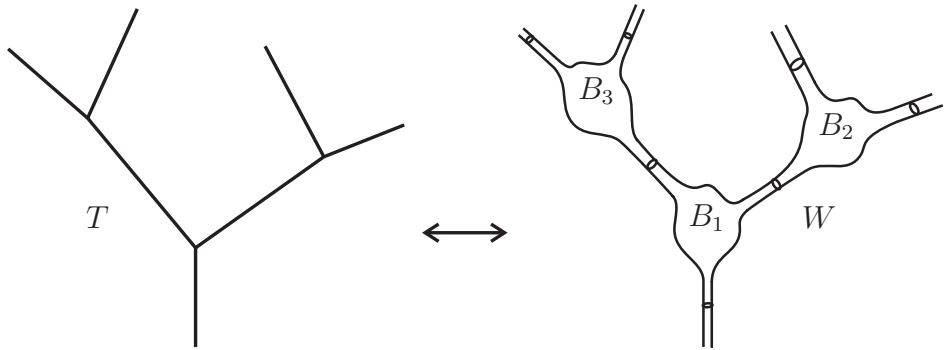}
\caption{A connected sum patterned on the tree $T$.}\label{fig:sum}
\end{figure}

\begin{defi} A block $B$ {\bf repeats finitely} in a manifold $W$ (both of dimension $d$) if $W$ contains a deleted $B$-block, but $W$ does not contain an infinite family of disjoint deleted $B$-blocks. The block $B$ {\bf repeats infinitely} if $W$ does contain an infinite family of disjoint deleted $B$-blocks.
\end{defi}

\noindent The following theorem, proven in \S\ref{nonleafproof}, is our
main result.

\begin{theoA}\label{theoremA}
Let $W$ be a sum-manifold patterned on an infinite tree
such that the fundamental group of each block is
generated by torsion elements of odd order (or is trivial)
and infinitely many non-homeomorphic blocks repeat finitely.
Then $W$ is not homeomorphic to any leaf of a codimension one foliation of a compact manifold.
\end{theoA}

In order to get blocks that repeat finitely, we shall use
sets of blocks with certain properties, which we now define.

\begin{defi} A set $\mathscr S$ of blocks of dimension $d\geq 3$ is {\bf non-repeating} if the two following conditions are satisfied:
\begin{enumerate}
\item If $B_1, B_2\in  \mathscr S$ are distinct blocks and
 $\pi_1(B_1)$
and $\pi_1(B_2)$ are expressed as free products of prime
groups, then $\pi_1(B_1)$ and $\pi_1(B_2)$ have no isomorphic
prime factors;
\item If a block $B_1\in {\mathscr S}$ is simply connected, then
there exist a number $r, 1<r
< d$ and a cyclic group ${\mathbb Z}_{p^j}$
of prime power order which is
isomorphic to a direct summand of
$H_r(B_1)$, but not to a direct summand of $H_r(B)$
for any other block $B\in {\mathscr S}$.
\end{enumerate} \label{nonrepset}
\end{defi}

\begin{prop}
If $W$ is a sum-manifold patterned on a tree using only blocks
from a non-repeating set, and a block $B$ is used for
only a finite number of vertices, then the block $B$
repeats finitely in $W$. \label{finiteblocks}
\end{prop}

This proposition will be proven in \S\ref{primeblocks}.
It should be noted that it is not obvious, since even if $B$
is an prime block, a deleted $B$-block
may occur in a sum-manifold in which no vertex was replaced
by a block homeomorphic to $B$, as in Example \ref{milnorex}, where
$S^2\times S^2\setminus {\rm Int}(D^4)$ is a deleted prime block
which embeds in $P\# P'\# P'$.
Theorem A and this proposition together imply the following result.

\begin{theoC} Let $W$ be a sum-manifold patterned on an infinite
tree using only blocks from a non-repeating set $\mathscr S$ such that
each block has a fundamental group generated by torsion
elements of odd order. If the set $\mathscr{T}\subseteq {\mathscr S}$
of blocks that are each used a finite non-zero number of times
to replace vertices of the tree in the construction of $W$ as a sum-manifold is infinite,
then the blocks in $\mathscr T$ repeat finitely, and consequently
$W$ is not homeomorphic to any leaf of a codimension one foliation of a compact manifold. \label{infrepeatfin}
\end{theoC}


\section{Prime blocks}\label{primeblocks}

In this section we prove Propositions \ref{indecsum}
and \ref{finiteblocks}, but first we
observe that many prime blocks exist.

\begin{prop}

1. A $3$-dimensional block $B$ whose fundamental group $\pi_1(B)$ is not isomorphic to a free product of non-trivial groups
(for example, a lens space) is prime.

2. For every dimension $d\geq 5$ and every positive integer $n$, there is a prime $1$-connected $d$-dimensional
block $B$ with $\pi_2(B)$ isomorphic to $Z_n$ if $d\geq 6$, or to $Z_n$ or $Z_n\oplus Z_n$ if $d=5$. \label{existblocks}
\end{prop}

Milnor \cite{M} remarks that
the first assertion follows easily from the Poincar\'e conjecture
that a simply connected closed connected $3$-manifold
is homeomorphic to the $3$-sphere. It has
been proven by G. Perelman. If a $3$-dimensional block is a
connected sum of two non-trivial blocks $B_1$ and $B_2$,
the Seifert-van Kampen Theorem shows that its
fundamental group must be isomorphic to the non-trivial free product
$\pi_1(B_1)*\pi_1(B_2)$.
The proof of the second assertion is given in \S\ref{pconstruct}.

In the proof of Proposition \ref{indecsum} we shall use the
Grushko-Neumann Theorem (\cite{K}, vol. 2, p. 58):
\begin{theo}{\ \rm (Grushko-Neumann)}
Every finitely generated group can be represented as a
free product of a finite number of groups that are indecomposable as
free products, and the decomposition is unique up to isomorphism
and the order of the factors.
\end{theo}

\begin{proof}[Proof of Proposition \ref{indecsum}.] Let $B$ be
a block of dimension $d$ that is homeomorphic
to $B_1\#\dots\# B_k$ for some $k$ with all blocks
$B_i$ non-trivial.
Since the Proposition is well-known for dimensions $1$ and $2$, we
assume that $d\geq 3$. Reorder the factors
so that the blocks $B_i$ for $1\leq i\leq \ell$ are the
ones which are not simply connected.
Now $\pi_1(B)$ is isomorphic to the free product
$\pi_1(B_1)*\dots *\pi_1(B_\ell)$ by the
Seifert-van Kampen Theorem. Since these groups are all
finitely generated, the Grushko-Neumann Theorem shows
that each has a unique decomposition as a free product of a finite number of groups that are indecomposable as free products.
If $\pi_1(B)$ has $m$ free factors, then $\ell\leq m$,
since each $\pi_1(B_i)$ contributes at least one free factor.

Now the direct sum of the homology groups
$\bigoplus_{r=2}^{d-1}H_r(B)$ is a finitely generated
abelian group, so it is a direct sum of a finite
number, say $n$, of cyclic groups of infinite or prime power
order. If $B$ is a connected sum of two blocks, so that
$B$ is the union of deleted blocks $B'$ and $B''$ with $B'\cap B''=S^{d-1}$, then
it follows from the Mayer-Vietoris exact
sequence of homology groups
$$H_{r-1}(S^{d-1})\to H_r(B')\oplus H_r(B'')\to H_r(B)\to H_r(S^{d-1})$$
that $H_r(B)\approx H_r(B')\oplus H_r(B'')$
for $1<r<d$, since the first
and last terms vanish. By induction it follows that
$H_r(B)= \bigoplus_{i=1}^kH_r(B_i)$ for every $r, 1<r<d$, so at most $n$ blocks $B_i$ can have $H_r(B_i)\neq 0$ for some $r,\ 1< r
<d$. By the generalized Poincaré conjecture, every
$B_i$ with $i>\ell$ must have such a non-trivial homology
group, for otherwise it would be a $d$-sphere, which
is excluded. Thus $k-m\leq n$ and $k$ cannot be greater than $m+n$.
If we choose the connected sum decomposition so that $k$
is maximal, then each summand $B_i$ must be prime.
\end{proof}

We shall use the following result.

\begin{prop} Any compact set $K$ contained in a sum-manifold $W$
must be contained in the union of a finite number of the
deleted blocks whose union is $W$. \label{cptfinite}
\end{prop}

\begin{proof} Let the deleted $B_n$-blocks whose union is $W$ be $B'_n, n\in{\mathbb N}$, so that $W=\cup_{n=1}^\infty B'_n$ and the interiors of the $B'_n$'s are disjoint.
Let $A_n$ be the interior of $B'_1\cup\dots\cup B'_n$ in $W$. The sets $A_n$ form a nested open cover of $K$ in $W$. Since $K$ is compact there exists an $n_0\geq 1$ such that $K$ is contained in $A_{n_0}$, but no $B'_n$ with $n>n_0$ can meet $A_{n_0}$.
\end{proof}

\begin{lemma}
If a deleted block $B$ of dimension $d\geq 3$  is contained in a
connected $d$-manifold $M$, then $\pi_1(B)$ is isomorphic
to a free factor of $\pi_1(M)$, i.e., there exists a finitely
generated group $H$ such that $\pi_1(M)\approx \pi_1(B)*H$.
\label{lemmafactor}\end{lemma}

\begin{proof}
Note that $\partial B = S_1\sqcup\dots\sqcup S_\ell$ is a disjoint
union of $(d-1)$-spheres. Let $B, L_1,\dots, L_r$ be the
connected components obtained from $M$ by cutting
it along the spheres $S_1,\dots, S_\ell$. Let us glue
the pieces $L_i$ to $B'$ successively along $r$ of the
spheres to obtain a connected manifold $L$, with the
remaining spheres (if any) as boundary components of $L$.
Observe that repeated application of the Seifert-van Kampen
Theorem shows that $\pi_1(L)\approx\pi_1(B)*G$ where $G=\pi_1(L_1)*\dots*\pi_1(L_r)$.

Now glue the remaining boundary spheres (if any) together to
obtain $M$. If there remain $s$ pairs of spheres to
be identified to obtain $M$ from $L$, then
$\pi_1(M)\approx \pi_1(L)*{\mathbb Z}^s$, since
each pair that is glued changes the fundamental group
by forming a free product with ${\mathbb Z}$.
Then $\pi_1(M)\approx \pi_1(B)*G*{\mathbb Z}^s$, as desired.
\end{proof}

\begin{lemma} If $B_1,\dots, B_k$ are disjoint
$d$-dimensional deleted blocks
contained in a manifold $W$ of dimension $d\geq 3$
(possibly with boundary)
then  $\pi_1(W)$ is isomorphic to the free product
of two subgroups, $\pi_1(W)=G*H$, with
$G\approx\pi_1(B_1)*\dots *\pi_1(B_k)$.
\label{finiteblocks2}\end{lemma}

The proof of this lemma is similar to that of the previous lemma
and is omitted. The following lemma is proven in a similar way,
using the Mayer-Vietoris exact sequence as in the proof of
Proposition \ref{indecsum}, instead of the
Seifert-van Kampen Theorem.

\begin{lemma} If $B_1,\dots, B_k$ are disjoint
$d$-dimensional deleted blocks
contained in a manifold $W$ of dimension $d\geq 3$
(possibly with boundary) and $2\leq r<d$, then
then  $H_r(W)$ is isomorphic to the direct sum
of two subgroups, $H_r(W)=G\oplus H$, with
$G\approx H_r(B_1)\oplus\dots\oplus H_r(B_k)$.
\label{finiteblocks3}\end{lemma}

\begin{proof}[Proof of Proposition \ref{finiteblocks}.] Let $B'_1, \dots, B'_\ell$ be disjoint
deleted non-trivial $B$-blocks contained in a sum-manifold $W$
patterned on a tree $T$ and suppose that only $k$
deleted $B$-blocks are used in constructing $W$ and
all the blocks used are taken from a non-repeating set.
By Proposition \ref{cptfinite}, the compact disjoint union
$B'_1\sqcup \dots\sqcup B'_\ell$ is contained in
a compact sum-manifold $W_1$ patterned on a
finite subtree $T_1$ using $k_1\leq k$ deleted $B$-blocks.
Let $v_1,\dots,v_n$ be all the
vertices of $T_1$, and let $B_i$ replace the vertex $v_i$ in the construction of $W_1$
We distinguish two cases.

Case 1. $\pi_1(B)\neq 1$. By the Seifert-van Kampen Theorem $\pi_1(W_1)$
is the free product $\pi_1(B_1)*\dots *\pi_1(B_n)$.
By property (1) of Definition
\ref{nonrepset}, some indecomposable factor of
$\pi_1(B)$ appears only $k_1$ times in the free product
decomposition of $\pi_1(W_1)$ into indecomposable factors.
On the other hand, by Lemma \ref{finiteblocks2}, this
group must appear at least $\ell$ times, so $\ell\leq k_1
\leq k$.

Case 2. $\pi_1(B)$ is trivial. Since the simply connected
block $B$ is non trivial, by property (2)
of Definition \ref{nonrepset},
for some dimension $r$ with $2\leq r <d$,
$H_r(B)$ is non-trivial and contains a non-trivial direct summand
${\mathbb Z}_{p^s}$, so that no other block in the
non-repeating set has a summand isomorphic
${\mathbb Z}_{p^s}$ in its $r$th
homology. By Lemma \ref{finiteblocks3},
$H_r(W_1)$ contains at least $\ell$ direct summands isomorphic
to ${\mathbb Z}_{p^s}$.
Now by the Mayer-Vietoris exact sequence
$H_r(W)$ is the direct sum of the
$r$th homology groups of the blocks used in
constructing $W_1$. Hence the group ${\mathbb Z}_{p^s}$
occurs exactly $k_1$ times in the direct sum decomposition
of $H_r(W)$ into cyclic groups, so again $\ell\leq k_1\leq k$.
\end{proof}


\section{Manifolds that are non-per\-iodic in homotopy and homology}\label{nonpersec}

\begin{defi} A  $(k-1)$-connected manifold $M$ is {\bf non-periodic in homotopy} in dimension $k\geq 2$ if its $k$th homotopy group $\pi_k(M)$ is isomorphic to the direct sum of
cyclic groups of odd prime power order, where for an infinite
number of odd prime powers $p^k$, the number of summands
of order $p^k$ is finite but non zero.
A $k$-manifold is
{\bf non-periodic in homology} in dimension $k\geq 2$ if its $k$th homology group $H_k(M)$ is isomorphic to a direct sum of cyclic groups satisfying the same property.
\label{nonperiodic}
\end{defi}

Using blocks obtained in Proposition \ref{existblocks} it is
easy to construct many $(k-1)$-connected sum-manifolds
that are non-periodic in homotopy (and in homology) of dimension $k\geq 2$.

\begin{proof}[Proof of Proposition \ref{existnonper}.] Let $(p_n)_{n\in \Nset}$ be a sequence of prime numbers
such that infinitely many primes occur finitely many
times in the sequence and let $T$ be a countable tree with vertices $\{v_n\}_{n\in\Nset}$. Given $d\geq 5$, let $B_n$ ($n\in {\mathbb N}$) be blocks of dimension $d$
such that $\pi_2(B_n)$ is ${\mathbb Z}_{p_n} \oplus {\mathbb Z}_{p_n}$
or ${\mathbb Z}_{p_n}$, as obtained in Proposition \ref{existblocks}. Now $H_2(B_n)\approx \pi_2(B_n)$ by the Hurewicz Theorem. Since the deleted blocks are identified along $(d-1)$-spheres with $d\geq 5$, the Mayer-Vietoris sequence shows that
$H_2(W)$ is the direct sum of the groups $H_2(B_n)$ and this sum is isomorphic to $\pi_2(W)$. (This argument is given in more
detail in the proof of Proposition \ref{recorrence} below.)

The $d$-manifold $W$ constructed as a connected sum patterned on $T$ using the blocks $B_n$ is clearly non-periodic in homotopy and in homology. \end{proof}

\begin{prop} Sum-manifolds that are non-periodic in homotopy or in homology satisfy the hypotheses of Theorem A.\label{recorrence}
\end{prop}
The proof of this proposition will be given in the next Section. Together with Theorem A it establishes Theorem B.


\section{Construction of blocks and proof of Proposition \ref{existblocks}}\label{pconstruct}

In this section we construct simply connected blocks with
given homology groups in dimension $2$ to prove the second
part of Proposition \ref{existblocks}.

\begin{proof}[Proof of Proposition \ref{existblocks}.]
In case the dimension $d$ is $5$, the main theorem of Smale \cite{Sm} states that for every finitely generated free abelian group $F$ and finite abelian group $G$ there is a $1$-connected smooth closed $5$-manifold $B$ with $\pi_2(B)\approx F\oplus G\oplus G$. If we take $F$ to be trivial and
set $G={\mathbb Z}_p$ then $\pi_2(B)\approx {\mathbb Z}_p\oplus {\mathbb Z}_p$, as desired.

For $d\geq 6$ let $f:S^1=\partial D^2\to \Sset^1$ be a smooth non-singular map of degree $p$,
where $D^m$ denotes the closed $m$-dimensional disk,
so that $H_1(\Sset^1\cup_fD^2)\approx {\mathbb Z}_p$.
The $2$-dimensional complex $\Sset^1\cup_fD^2$ can be embedded in ${\mathbb R}^5$ and so its suspension $K$ can be considered a subset of $\Rset^d$,
$$K=\Sset^2\cup_{Sf}  D^3 = S(\Sset^1\cup_{f}  D^2)\subset \Rset^6\subseteq\Rset^d.$$
Take a closed tubular neighborhood $V$ of $K$ in $\Rset^d$ and identify the boundaries of two copies of $V$, say $V_1$ and $V_2$, to obtain the double of $V$, a compact $1$-connected $d$-manifold which we denote by $B$. Note that $\partial V$ is $1$-connected, since a loop in
$\partial V$ has a contraction in $V$, and the contracting singular surface can be pushed off $K$ and into $\partial V$. The inclusion $\partial V\hookrightarrow V_i$ induces a surjective homomorphism on the second homology $H_2(\partial V)\to H_2(V_i), i=1,2$, for the same dimensional reason.
Then the Mayer-Vietoris sequence
$$\cdots\to H_2(\partial V)\to H_2(V_1)\oplus H_2(V_2)\to H_2(B)\to H_1(\partial V)
\to \cdots$$
shows that $H_2(B)\approx {\mathbb Z}_p$, so by the Hurewicz theorem, $\pi_2(B)\approx {\mathbb Z}_p$.

Now it is not clear that the blocks constructed with either
${\mathbb Z}_{p^k}$ or ${\mathbb Z}_{p^k}\oplus {\mathbb Z}_{p^k}$
as second homology are prime, but by Proposition \ref{indecsum}
they can be expressed as connected sums of a finite number
of prime blocks, one of which must have
$H_2(B)$ isomorphic to
${\mathbb Z}_{p^k}$ or ${\mathbb Z}_{p^k}\oplus {\mathbb Z}_{p^k}$.
\end{proof}

The proof of the Proposition can easily be adapted to show the existence of prime $(k-1)$-connected blocks of dimension $k+3$ or greater whose $k$th homotopy group is ${\mathbb Z}_{p^k}$.

\begin{proof}[Proof of Proposition \ref{recorrence}.] Let $W$ be a connected sum of blocks $B_n$, $n\in {\mathbb N}$, patterned on a tree $T$, such that $\pi_2(B_n)$ is isomorphic to $G_n$, where these groups satisfy the conditions of Definition \ref{nonperiodic}.
Let $i$ be one of the infinitely many indices for which $R_i=\{n\in \Nset;\ p_n=p_i\}$ is finite. When $d\geq 6$, take $k$ to be greater than the number of elements in $R_i$. We claim that $W$ cannot contain $k$ deleted blocks corresponding to $B_i$, so that  $B_i$ repeats finitely in $W$.

 Suppose, to obtain a contradiction, that $W$ contains $k$ disjoint submanifolds with boundary, say $W_1,\ldots, W_k$, each homeomorphic to $B_i$ minus a finite set of open $d$-balls, where $k$ is greater than the number of elements of $R_i$. Let $Y$ be the closure of $W\setminus (W_1\cup \ldots \cup W_k)$. Thus, $Y\cap (W_1\cup\ldots \cup W_k)$ is a disjoint union of $(d-1)$-spheres. Applying the exact sequence of Mayer-Vietoris to $W=Y\cup (W_1\cup \ldots \cup W_k)$ we have:
\begin{displaymath}
0 \to H_2(Y)\oplus H_2(\stackrel{k}{\mathop{{\cup}}\limits_{l=1}}W_l) \stackrel{\varphi}{\mathop{\to}} H_2(W)
\end{displaymath}
so the homomorphism $\varphi$ is injective. In particular, $$H_2(\stackrel{k}{\mathop{{\cup}}\limits_{l=1}} W_l)\simeq    \underbrace{\Zset_{p_i}\oplus \ldots \oplus \Zset_{p_i}}_{k-\textrm{times}}$$ injects into $H_2(W)$, but that is a contradiction because $H_2(W)$ has less than $k$ copies of $\Zset_{p_i}$.

For the case $d=5$, each group $H_2(B_n)$ with $n\in R_i$ is
isomorphic to ${\mathbb Z_{p_i}}\oplus {\mathbb Z_{p_i}}$, so a similar argument works if we take
$k$ greater than twice the number of elements of $R_i$.
\end{proof}

We note that for $k\geq 3$ there is an analogous construction for
 $(k-1)$-connected blocks $B$ of dimension  $d\geq k+4$ with $\pi_k(B)\approx {\mathbb Z_p}$ by using the $(k-1)$st suspension of the embedding
 $\Sset^1\cup_{f}  D^2\hookrightarrow \Rset^5$. We obtain
open $(k+4)$-manifolds that are non-periodic in homotopy (and in
homology) of dimension less than or equal to $k$, possibly mixing homotopy
and homology groups in distinct dimensions. It is easy to adapt the proof of Theorem B to show that manifolds that are non-periodic in homotopy
and homology of higher dimensions cannot be codimension one leaves in compact manifolds.


\section{Non-leaves of foliations of codimension one}\label{nonleafproof}

In this section we give the proof of Theorem A which implies  Theorems B and C. The proof is a small adaptation of the proof of Ghys' main theorem in \cite{Gh} (which it generalizes), although the final arguments, involving blocks that repeat finitely, are simpler.

Throughout this section we
let $L$ be a manifold that satisfies the hypotheses of
Theorem A, so it is a connected sum of blocks
patterned on a countable tree. Furthermore infinitely many blocks repeat finitely and every block has a fundamental group generated by torsion elements with no $2$-torsion (possibly the block may be simply connected).
We suppose that $L$ is a leaf of a $C^0$ foliation  $\Fol$  of codimension one of a compact manifold $M$, in order to arrive at a contradiction.

Since every block has a fundamental group generated by torsion elements with no $2$-torsion (or possibly trivial), each block must
have trivial holonomy in $\Fol$, since every torsion element of odd order has trivial holonomy and then the holonomy of the whole leaf $L$ must be trivial.
We may assume that $\Fol$ is transversely oriented, by passing to a double cover if necessary. Under these conditions the leaf $L$ is proper. In fact, as in Lemma 4.3 of Ghys \cite{Gh}, we can take a block of $L$ which repeats finitely and as the holonomy of $L$ is trivial the Reeb Stability Theorem provides us with an embedding of a product neighborhood of this block on $M$. Thus, each transversal section induced by the embedding intercepts the leaf $L$ in a
finite number of points.

Fix a one-dimensional foliation $\Nfol$ transverse to $\Fol$
(see \cite{HH}, Theorem 1.1.1, pp. 2-3).
Given an open saturated set $U\subset M$, there is a manifold $\widehat{U}$ with boundary and corners, called the {\bf completion} of $U$, and an immersion $i:\widehat{U}\to M$ such that $U$ is the interior of $\widehat U$ and $i$ restricted to $U$ is the inclusion of $U$ in $M$ (see \cite{HH} pp. 87--88 for the explicit construction of $\widehat{U}$). The foliations $i^*\Fol$ and $i^*\Nfol$  agree with $\Fol$ and $\Nfol$ on $U$.

\begin{theo}[Dippolito \cite{Di}, Hector \cite{HH}]
Under the above hypotheses, there is a compact manifold with boundary and corners $K$ in $\widehat{U}$ such that $\partial K=\partial^{tg}\cup \partial^{tr}$ where
\begin{enumerate}
\item $\partial^{tg}\subset \partial \widehat{U}$;
\item $\partial^{tr}$ is saturated by the foliation $i^*\Nfol$;
\item the complement of the interior of $K$ in $\widehat{U}$ is a finite union of connected non-compact submanifolds $B_i$ with boundary and corners, and there are non-compact manifolds with boundary $S_i$ so that each $B_i$ is homeomorphic to $S_i\times [0,1]$ by a homeomorphism $\phi_i:S_i\times [0,1]\to B_i$ that takes $\{\ast\}\times [0,1]$ to a leaf of $i^*\Nfol$.
\end{enumerate}\end{theo}

The compact set $K$ is the called the {\bf kernel} of $\widehat{U}$ and the submanifolds $B_i$ are the {\bf branches} of $\widehat{U}$. The foliation $i^*\Fol$ restricted to a branch $B_i$ is defined by the suspension of a representation of the fundamental group of $S_i$ into the group of orientation preserving homeomorphisms of the interval $[0,1]$.

\begin{figure}[H]
\centering
 \psfrag{B1}{$B_1$}
 \psfrag{B2}{$B_2$}
 \psfrag{S1}{$S_1$}
 \psfrag{S2}{$S_2$}
 \psfrag{K}{$K$}
\includegraphics*[width=.7\linewidth]{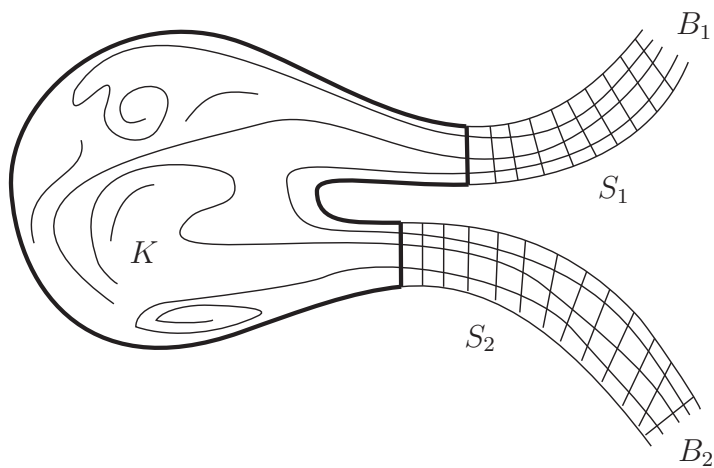}
\caption{The completion of a saturated open set.}\label{fig:completion}
\end{figure}

\begin{prop} (An extension of Reeb stability) Let $L$ be a leaf in the boundary $\partial\widehat U$, where $U$ is a connected saturated
open set of a transversely oriented codimension one foliation
of a compact manifold. If $L$ is a connected sum of
blocks (a sum-manifold as in Definition \ref{blocks}),
and each block has a fundamental
group that is trivial or generated by torsion, then $L$ has a
one-sided product foliated neighborhood in $\partial\widehat U$.
\label{extendedRS} \end{prop}

\begin{proof} Note that only a finite
number of blocks can meet a neighborhood of the kernel $K$, according to Proposition
\ref{cptfinite}. Hence we can enlarge the nucleus $K$ a
little so that $L\cap K$ is a union of a finite
number of deleted blocks. Then for each branch $B_i$, the
intersection $B_i\cap L$ is also a connected sum of blocks, so its
fundamental group is trivial or generated by torsion elements.
Since the holonomy of $\Fol |_{B_i}$ is globally defined, the foliation $\Fol |_{B_i}$ is a product foliation.
Similarly, by the usual Reeb stability theorem,
the compact set $L\cap K$ has trivial holonomy and
therefore has a one-sided product foliated neighborhood in $K$.
Combining the product foliated neighborhoods in the kernel
$K$ and the blocks $B_i$ gives a one-sided product foliated
neighborhood of $L$ in $\widehat U$.
\end{proof}

\begin{lemma}\label{localprod}
The leaf $L$ possesses an open neighborhood saturated by $\Fol$ which is homeomorphic to $L\times (-1,1)$ by a homeomorphism that takes $L\times \{\ast\}$ to leaves of $\Fol$ and $\{\ast\}\times (-1,1)$ into leaves of $\Nfol$.
\end{lemma}

\begin{proof} The proof uses Dippolito's Theorem and follows the proof of Lemma 4.4 of Ghys \cite{Gh}. Let $\tau: [0,1)\to M$ be a positive transversal that meets $L$ only in $\tau(0)$, and let $U$ be the saturation of $\tau(0,1)$. Then one of the leaves in the boundary of the completion $\widehat U$
is homeomorphic to $L$ and has trivial holonomy, so by Proposition \ref{extendedRS}, $L$ has a one-sided saturated product neighborhood. The same argument applies to the other side of $L$ to give the desired product neighborhood.
\end{proof}

\begin{proof}[Proof of Theorem A]
Let $\Omega$ be the union of all the leaves of $\Fol$ that are homeomorphic to $L$ and let $\Omega_1$ be the connected component of $\Omega$ that contains $L$. The previous lemma
shows that every leaf in $\Omega$ has an $\Fol$-saturated open neighborhood foliated as a product by the two foliations $\Fol$ and $\Nfol$.
These product neighborhoods fit together to show that $\Omega_1$
is an open saturated set that fibers over a Hausdorff manifold of dimension one, either an open interval or a circle,
with the leaves of $\Fol$ as fibers.  Consider the completion $\widehat\Omega_1$.

The proof of Lemma 4.5 of Ghys \cite{Gh} shows the following result.

\begin{lemma} \label{lemma:Ghys}
The set $\widehat{\Omega}_1$ is noncompact and $\partial \widehat{\Omega}_1$ is nonempty.
Every leaf of $\Fol$ contained in $\partial \widehat{\Omega}_1$ has an infinite cyclic holonomy group generated by a contraction.
\end{lemma}
\begin{proof}
The set $\Omega_1$ is not compact. In fact, if $\Omega_1$ is compact, then $\Fol|_{\Omega_1}$ possesses a minimal set $\mu$. Let $F_0$ be a leaf of $\Fol|_{\Omega_1}$ in $\mu$. Then $F_0$ is dense in $\mu$, but given a point $x\in F_0$ there is a neighborhood $V$ of $x$ in $M$ that intercepts only one plaque of the leaf $F_0$, because $F_0$ is proper. Then $F_0$ is open in $\mu$, so $\mu\setminus F_0$ is closed and therefore $\mu\setminus F_0=\varnothing$. Then  $F_0$ is closed and therefore compact, but this is a contradiction. Therefore $\Omega_1$ is not compact and $\partial \widehat{\Omega}_1$ must be nonempty.

Let $F$ be a leaf contained in $\partial \widehat{\Omega}_1$ and
suppose that the holonomy group of $F$ is trivial, to get a contradiction. Since the leaves in $\Omega$ are homeomorphic
under translation along the transverse foliation $\Nfol$,
any loop in $F\cap B_i$ must have globally trivial holonomy
in $B_i$. In $F\cap K$ the compact set $F\cap N$ will have a
one-sided product foliated neighborhood, and hence $F$ must have a product neighborhood in $\widehat{\Omega}_1$ and will be homeomorphic to $L$, but since $F$ belongs to the boundary of $\widehat{\Omega}_1$ it is not homeomorphic to $L$. Hence the holonomy of $F$ must be
nontrivial. The leaves in the interior of $\widehat\Omega_1$ are proper, so the holonomy group of $F$ acts discretely; it must be infinite cyclic and one of its two generators will be a contraction.
Hence there is a small open neighborhood $V$ of $\partial \widehat\Omega_1$ in $\widehat\Omega_1$ such that every block of $L$ contained in $V$ repeats infinitely.

Suppose that $\widehat\Omega_1$ is compact and let $C=\widehat\Omega_1\setminus V$. Then $C\cap L$ is closed in $\widehat{\Omega}_1$, so it is a compact set contained in $L$. The proper leaf $L$ has the induced topology and so $C\cap L$ is compact in the topology of $L$.
By Proposition \ref{cptfinite} only finitely many blocks of $L$ meet $C\cap L$. Therefore some block of $L$ that repeats finitely in $L$ is contained in $V$, so it must repeat infinitely, which gives a contradiction.
\end{proof}

If $\Omega_1$ were fibered over an open interval
$(0,1)$, then it would be homeomorphic to a foliated product $L\times (0,1)$ and the completion would be $\widehat{\Omega}_1\approx L\times [0,1]$, which is impossible since the boundary leaves of
$\widehat{\Omega}_1$ are not homeomorphic to $L$.

Hence $\Omega_1$ must fiber over the circle. The remainder of our proof simplifies considerably the argument of Ghys in Section 5 of \cite{Gh}. Let $h:L\to L$ be the monodromy map that takes a point $x$ in $L$ to its first return to $L$ in the positive direction along the leaf of $\Nfol$ that contains $x$. The local product structure given by Lemma \ref{localprod} shows that $h$ is defined
on all of $\Omega_1$.
If a point $x$ of $L$ is in a branch $B_i$, then $h(x)>x$ on the interval $\{*\}\times[0,1]$ in the leaf of
$\Nfol$ that contains $x$, so $x$ is neither fixed nor periodic.
If $x$ is in a sufficiently small open neighborhood $V$ of $\partial\widehat \Omega_1$,
then again $h(x)>x$ and $x$ is neither fixed nor periodic.
Set $K'=K\setminus V$ and observe that
$K'\cap L$ is compact, so by Proposition \ref{cptfinite}, $K'$ meets finitely many blocks of $L$. Hence some block $C$ that repeats finitely on $L$ must be contained in $L\cap (V\cup \bigcup B_i)$, which contains no periodic points of $h$.
Therefore there is an integer $r>0$ such that the compact sets $h^{nr}(C)$, $n\in \Nset$, are pairwise disjoint, so $C$ must repeat infinitely. This contradiction completes the proof of Theorem A. \end{proof}

The manifold constructed by Ghys \cite{Gh} and similar constructions patterned on an arbitrary (locally finite) tree cannot be homeomorphic to leaves in codimension one; this follows from Theorem C. Inaba {\it et al.}  \cite{INTT} show that the manifolds that they construct are not $C^2$ diffeomorphic to leaves in codimension one, but Theorem C shows that they are not even homeomorphic to leaves in codimension one.
Theorem C also gives examples of non-leaves that mix blocks of various types---with finite fundamental groups or varying higher homotopy groups---provided that infinitely many blocks repeat finitely.


\end{document}